\documentclass[10pt]{article}
\usepackage{amsmath,amssymb,amsfonts,amsthm,mathtools,epsfig}
\usepackage{tikz}
\usepackage{caption,subcaption,verbatim}
\usepackage{setspace}
\usepackage{algorithm}
\usepackage{algorithmic}
\usepackage{xurl}
\usepackage{hyperref}
\usepackage{ulem}
\usepackage{thmtools}
\usepackage[T1]{fontenc}
\usepackage{array}
\usepackage{makecell}

\newcolumntype{C}[1]{>{\centering\arraybackslash}p{#1}}

\usepackage{cancel}

\usepackage[bottom=1.5in]{geometry}

\tikzset{
    graphvertex/.style={circle, fill=black, inner sep=1.6pt},
    graphedge/.style={line width=0.45pt}
}

\usetikzlibrary{arrows.meta,positioning,calc,decorations.pathreplacing}

\newtheorem{thm}{Theorem}
\newtheorem{theorem}[thm]{Theorem}
\newtheorem{lemma}[thm]{Lemma}

\theoremstyle{definition}

\newtheorem{problem}{Problem}
\newtheorem{question}[problem]{Question}

\date{ }

\title{\bf Vertex-critical $(P_5,\text{chair})$-free and $(P_5,\text{cricket})$-free graphs}

\author{
Jorik Jooken \thanks{Department of Computer Science, KU Leuven Kulak, 8500 Kortrijk, Belgium.\\ Email address:
 \protect\href{mailto:jorik.jooken@kuleuven.be}{\protect\nolinkurl{jorik.jooken@kuleuven.be}}
}}

\date{}

\begin{document}
	
	\maketitle

	 \begin{abstract}
For graphs $G, F_1$ and $F_2$, we say that $G$ is $(F_1,F_2)$-free if neither $F_1$ nor $F_2$ is an induced subgraph of $G$. We say that $G$ is $k$-vertex-critical if the chromatic number of $G$ is $k$, but every proper induced subgraph of $G$ has chromatic number at most $k-1$. The $\text{chair}$ graph is a $5$-vertex graph obtained by adding a pendant vertex to one of the two central vertices of a path on $4$ vertices. The $\text{cricket}$ graph is a $5$-vertex graph obtained by adding two pendant vertices to a common vertex of a triangle. The path on $5$ vertices is denoted by $P_5$. We prove that for every $k \geq 1$, there are only finitely many $(P_5,\textit{chair})$-free $k$-vertex-critical graphs. We also prove that the same conclusion holds if $\textit{chair}$ is replaced by $\textit{cricket}$. We further characterize all $5$-vertex-critical $(P_5,\text{chair})$-free graphs, all $5$-vertex-critical $(P_5,\text{cricket})$-free graphs and all $6$-vertex-critical $(P_5,\text{cricket})$-free graphs. Our proofs rely on bounding the size of antichains and developing Ramsey-theoretic ideas. For any fixed integer $k \geq 1$, our results imply the existence of a polynomial time algorithm to decide whether a $(P_5,\textit{chair})$-free (or $(P_5,\textit{cricket})$-free) graph is $(k-1)$-colourable such that this algorithm can also present a negative constant-size certificate in case the graph is not $(k-1)$-colourable.

			\vskip 3mm
		
			\end{abstract}

\section{Introduction}

 This paper is only concerned with finite, simple graphs. For graphs $G$ and $F$, we say that $G$ is \textit{$F$-free} if $F$ does not occur as an induced subgraph of $G$. Similarly, for a family of graphs $\mathcal{F}$, we say that $G$ is \textit{$\mathcal{F}$-free} if $G$ is $F$-free for every graph $F \in \mathcal{F}$. A \textit{proper $k$-colouring} of $G$ maps every vertex in $G$ to a label from the set $\{1,2,\ldots,k\}$ such that adjacent vertices receive a different label. The smallest integer $k$ for which $G$ admits a proper $k$-colouring (we say $G$ is \textit{$k$-colourable}) is the \textit{chromatic number of $G$} (denoted by $\chi(G)$). We say that $G$ is \textit{$k$-vertex-critical} if $\chi(G)=k$, but $\chi(G-v)<k$ for every vertex $v \in V(G)$. For example, an odd cycle is $3$-vertex-critical.

There exists a large body of literature concerned with the following difficult question.
\begin{question}
\label{ques:mainQuestion}
Let $k \geq 4$ be an integer and let $\mathcal{F}$ be a family of graphs. Are there only
finitely many $k$-vertex-critical $\mathcal{F}$-free graphs?
\end{question}

This question has a direct algorithmic motivation. More precisely, if there exist only
finitely many $k$-vertex-critical $\mathcal{F}$-free graphs, then there is a polynomial time algorithm to decide whether an $\mathcal{F}$-free graph $G$ is $(k-1)$-colourable and this polynomial time algorithm can present a negative certificate in case $G$ is not $(k-1)$-colourable. Indeed, $G$ is $(k-1)$-colourable if and only if $G$ is $H$-free for each of the finitely many $k$-vertex-critical $\mathcal{F}$-free graphs $H$ and this can be checked in polynomial time.

Question~\ref{ques:mainQuestion} is far from being solved in general. The case where $\mathcal{F}=\{F\}$ (i.e., there is only one forbidden induced subgraph) is a natural starting point. In this case, the question has been completely solved for $k=4$. Indeed, Chudnovsky et al.~\cite{CGSZ20a} first solved this case for connected graphs $F$ and later the same authors~\cite{CGSZ20b} solved the question for all graphs $F$: there are only finitely many $4$-vertex-critical $F$-free graphs if and only if $F$ is an induced subgraph of $P_6, 2P_3$ or $P_4 + \ell P_1$ for some integer $\ell \geq 0$. Here, $P_t$ denotes the path on $t$ vertices, $G_1+G_2$ denotes the disjoint union of $G_1$ and $G_2$, and $\ell G_1$ denotes the disjoint union of $\ell$ copies of $G_1$. For $\mathcal{F}=\{F\}$ and $k \geq 5$, it follows from the work of Erd{\H{o}}s~\cite{E59}, Abuadas et al.~\cite{ACHS24}, Cameron et al.~\cite{CHS22} and Ho{\`a}ng et al.~\cite{HMRSV15} that the only open cases are those where $F=P_4+\ell P_1$ for some integer $\ell \geq 1$. A particular result that is important for the current paper is that there are only finitely many $4$-vertex-critical $P_5$-free graphs, but infinitely many $k$-vertex-critical $P_5$-free graphs for each integer $k \geq 5$~\cite{HMRSV15}.

There are many more open questions for the case where $\mathcal{F}=\{F_1,F_2\}$ (i.e., there are two forbidden induced subgraphs). Here, we write that $G$ is $(F_1,F_2)$-free instead of $\mathcal{F}$-free. A particularly well-studied line of work concerns the case where $\mathcal{F}=\{P_5,F_2\}$. Cameron, Goedgebeur, Huang and Shi~\cite{CGHS21} initiated a systematic
study of these classes and obtained a complete dichotomy when $F_2$ has order
four: for every integer $k \geq 5$, there are infinitely many $k$-vertex-critical $(P_5,F_2)$-free graphs if and only if $F_2 \in \{2P_2,K_3+P_1\}$. Since there are only finitely many $4$-vertex-critical $P_5$-free graphs~\cite{HMRSV15}, this completely settles the $(P_5,F_2)$-free case where $F_2$ has order $4$. This led these authors to ask the following natural and significantly more challenging question in~\cite{CGHS21}.

\begin{question}
\label{ques:OrderFiveQuestion}
Let $k \geq 5$ be an integer and let $F_2$ be a graph of order $5$. Are there only
finitely many $k$-vertex-critical $(P_5,F_2)$-free graphs?
\end{question}

\begin{figure}[ht]
\centering

\begin{subfigure}{0.18\textwidth}
\centering
\begin{tikzpicture}[scale=0.8, baseline=(current bounding box.center)]
    % chair
    \node[graphvertex] (a) at (1.0,-0.9) {};
    \node[graphvertex] (b) at (0,0.9) {};
    \node[graphvertex] (c) at (0,0) {};
    \node[graphvertex] (d) at (1.0,0) {};
    \node[graphvertex] (e) at (0,-0.9) {};
    \draw[graphedge] (b)--(c)--(d);
    \draw[graphedge] (c)--(e);
    \draw[graphedge] (a)--(d);
\end{tikzpicture}
\caption{chair}
\end{subfigure}
\hfill
\begin{subfigure}{0.18\textwidth}
\centering
\begin{tikzpicture}[scale=0.8, baseline=(current bounding box.center)]
    % cricket
    \node[graphvertex] (a) at (0,0) {};
    \node[graphvertex] (b) at (-0.8,0.9) {};
    \node[graphvertex] (c) at (0.8,0.9) {};
    \node[graphvertex] (d) at (-0.8,-0.9) {};
    \node[graphvertex] (e) at (0.8,-0.9) {};
    \draw[graphedge] (a)--(b)--(c)--(a);
    \draw[graphedge] (a)--(d);
    \draw[graphedge] (a)--(e);
\end{tikzpicture}
\caption{cricket}
\end{subfigure}
\hfill
\begin{subfigure}{0.18\textwidth}
\centering
\begin{tikzpicture}[scale=0.8, baseline=(current bounding box.center)]
    % gem: P_4 plus a universal vertex
    \node[graphvertex] (a) at (0,1.2) {};
    \node[graphvertex] (b) at (1,1.5) {};
    \node[graphvertex] (c) at (2,1.5) {};
    \node[graphvertex] (d) at (3,1.2) {};
    \node[graphvertex] (e) at (1.5,0) {};
    \draw[graphedge] (a)--(b)--(c)--(d);
    \draw[graphedge] (e)--(a);
    \draw[graphedge] (e)--(b);
    \draw[graphedge] (e)--(c);
    \draw[graphedge] (e)--(d);
\end{tikzpicture}
\caption{gem}
\end{subfigure}
\hfill
\begin{subfigure}{0.18\textwidth}
\centering
\begin{tikzpicture}[scale=0.8, baseline=(current bounding box.center)]
    % banner: C_4 with a pendant vertex
    \node[graphvertex] (a) at (0,0) {};
    \node[graphvertex] (b) at (1.1,0) {};
    \node[graphvertex] (c) at (1.1,1.1) {};
    \node[graphvertex] (d) at (0,1.1) {};
    \node[graphvertex] (e) at (-0.8,-0.6) {};
    \draw[graphedge] (a)--(b)--(c)--(d)--(a);
    \draw[graphedge] (a)--(e);
\end{tikzpicture}
\caption{banner}
\end{subfigure}

\vspace{0.8em}

\begin{subfigure}{0.18\textwidth}
\centering
\begin{tikzpicture}[scale=0.8, baseline=(current bounding box.center)]
    % dart: diamond with a pendant vertex attached to a degree-3 vertex
    \node[graphvertex] (a) at (-0.9,0) {};
    \node[graphvertex] (b) at (0,0.9) {};
    \node[graphvertex] (c) at (0.9,0) {};
    \node[graphvertex] (d) at (0,-0.9) {};
    \node[graphvertex] (e) at (0,1.8) {};
    \draw[graphedge] (a)--(b)--(c)--(d)--(a);
    \draw[graphedge] (b)--(d);
    \draw[graphedge] (b)--(e);
\end{tikzpicture}
\caption{dart}
\end{subfigure}
\hfill
\begin{subfigure}{0.18\textwidth}
\centering
\begin{tikzpicture}[scale=0.8, baseline=(current bounding box.center)]
    % bull: triangle with two pendant horns facing upwards
    \node[graphvertex] (a) at (-0.8,0) {};
    \node[graphvertex] (b) at (0.8,0) {};
    \node[graphvertex] (c) at (0,-1.0) {};
    \node[graphvertex] (d) at (-0.8,0.8) {};
    \node[graphvertex] (e) at (0.8,0.8) {};
    \draw[graphedge] (a)--(b)--(c)--(a);
    \draw[graphedge] (a)--(d);
    \draw[graphedge] (b)--(e);
\end{tikzpicture}
\caption{bull}
\end{subfigure}
\hfill
\begin{subfigure}{0.18\textwidth}
\centering
\begin{tikzpicture}[scale=0.8, baseline=(current bounding box.center)]
    % K_5 - e: complete graph on five vertices with one edge omitted
    \node[graphvertex] (a) at (90:1.1) {};
    \node[graphvertex] (b) at (18:1.1) {};
    \node[graphvertex] (c) at (-54:1.1) {};
    \node[graphvertex] (d) at (-126:1.1) {};
    \node[graphvertex] (e) at (162:1.1) {};

    % all edges except a--b
    \draw[graphedge] (a)--(c);
    \draw[graphedge] (a)--(d);
    \draw[graphedge] (a)--(e);
    \draw[graphedge] (b)--(c);
    \draw[graphedge] (b)--(d);
    \draw[graphedge] (b)--(e);
    \draw[graphedge] (c)--(d);
    \draw[graphedge] (c)--(e);
    \draw[graphedge] (d)--(e);
\end{tikzpicture}
\caption{$K_5-e$}
\end{subfigure}

\caption{A visualization of several graphs on $5$ vertices.}
\label{fig:named-graphs}
\end{figure}
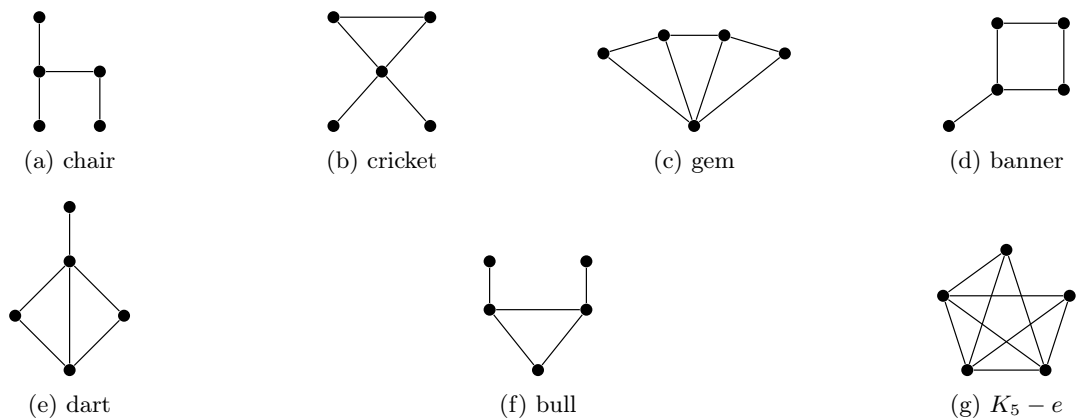

A large number of papers have made progress on this question, but several open cases remain. In particular, Question~\ref{ques:OrderFiveQuestion} is completely solved for the following cases of $F_2$: $5P_1$~\cite{R30}, $C_5$~\cite{CH23, HMRSV15}, $P_2+3P_1$~\cite{CHS22}, $P_3+2P_1$~\cite{ACHS24}, $K_{1,4}$~\cite{KP17}, $K_{2,3}$~\cite{KP17}, $\overline{K_3+2P_1}$~\cite{XJGH25a}, $K_{1,3}+P_1$~\cite{XJGH25a}, $\overline{P_3+P_2}$~\cite{CGH23}, gem~\cite{CGH23}, $\overline{P_5}$~\cite{DHHMMP17}, banner~\cite{BGS22}, dart~\cite{XJGH25b} and bull~\cite{BH26} (see Fig.~\ref{fig:named-graphs} for visualizations of several $5$-vertex graphs).

Moreover, Huang and Li~\cite{HL23} showed that there are only finitely many $5$-vertex-critical $(P_5,\text{chair})$-free graphs. Our first main result (Section~\ref{sec:chair}) extends their theorem: we prove that for all integers $k \geq 1$ there are only finitely many $k$-vertex-critical $(P_5,\text{chair})$-free graphs. Finally, we also characterize all $5$-vertex-critical $(P_5,\text{chair})$-free graphs, all $5$-vertex-critical $(P_5,\text{cricket})$-free graphs and all $6$-vertex-critical $(P_5,\text{cricket})$-free graphs in Section~\ref{sec:characterizations} using a computer-assisted approach.

By combining the previously discussed results, we can see that the only cases of Question~\ref{ques:OrderFiveQuestion} that remain open for some integers $k \geq 5$ are when $F_2$ is one of the following five graphs: $P_4+P_1$, $C_4+P_1$, $\overline{P_3+2P_1}$, $K_5-e$ and $K_5$.

\section{Vertex-critical $(P_5,\text{chair})$-free graphs}
\label{sec:chair}
Let $G$ be a graph, let $v \in V(G)$ be a vertex and let $X, Y \subseteq V(G)$ be two subsets of vertices. We write $N_X(v)$ to denote all neighbors of $v$ in $X$, i.e., $N(v) \cap X$. We write $G[X]$ to denote the graph induced by $X$. We say that $v$ is \textit{complete} (respectively, \textit{anticomplete}) to $Y$ if $v$ is adjacent (respectively, nonadjacent) to every vertex $y \in Y$. We say that $X$ is \textit{complete} (respectively, \textit{anticomplete}) to $Y$ if every vertex $x \in X$ is complete (respectively, anticomplete) to $Y$. We say that $v$ is \textit{mixed} on $X$ if $v$ is neither complete nor anticomplete to $X$. We say that $X$ is a \textit{homogeneous set} if no vertex in $V(G) \setminus X$ is mixed on $X$. 

We now recall two well-known lemmas from the literature on the structure of $k$-vertex-critical graphs.

\begin{lemma}[\cite{CGHS21}]
\label{lem:incompatible}
Let $G$ be a $k$-vertex-critical graph. Then $G$ has no two nonempty disjoint sets $X \subset V(G)$ and $Y \subset V(G)$ such that the following three conditions are simultaneously satisfied.
\begin{itemize}
    \item $X$ and $Y$ are anticomplete to each other.
    \item $Y$ is complete to $N(X)$.
    \item $\chi(G[X]) \leq \chi(G[Y])$.
\end{itemize}
\end{lemma}

\begin{lemma}[\cite{XJGH25b}]
\label{lem:homogeneous}
Let $G$ be a $k$-vertex-critical graph and let $S$ be a nonempty homogeneous set of $G$. For each connected component $A$ of $G[S]$, there exists an integer $m$ such that $1 \leq m<k$ and $A$ is an $m$-vertex-critical graph.  
\end{lemma}

Recently, Beaton and Cameron~\cite{BC26} showed the following result that will be important for us.

\begin{theorem}[\cite{BC26}]
\label{th:P4PlusLP1ChairFree}
    For all integers $k \geq 1$ and $\ell \geq 0$, there are only finitely many $k$-vertex-critical $(P_4+\ell P_1,\text{chair})$-free graphs.
\end{theorem}

Having introduced these prerequisites, we are now ready to prove the following theorem.
\begin{theorem}
\label{th:chair}
    For all integers $k \geq 1$, there are only finitely many $k$-vertex-critical $(P_5,\text{chair})$-free graphs.
\end{theorem}
\begin{proof}
    Our proof strategy is to show that for all integers $k \geq 1$, there exists an integer $f(k)$ such that every $k$-vertex-critical $(P_5,\text{chair})$-free graph is also $P_4+f(k) P_1$-free. If we can show this, the theorem follows from Theorem~\ref{th:P4PlusLP1ChairFree}.

    We will prove this statement by induction on $k$. If $k \leq 2$, the statement is trivial. Now let $k \geq 3$ and assume the statement holds for all integers $1 \leq m<k$. Let $G$ be a $k$-vertex-critical $(P_5,\text{chair})$-free graph. If $G$ is $P_4$-free, the statement trivially holds. 
    
    Otherwise, let $P := \{a,b,c,d\}$ be a 4-vertex subset of $V(G)$ such that $P$ induces the path $a-b-c-d$. Let $A := \{x \in V(G) \setminus P ~|~x\text{ is anticomplete to }P\}$. We will now work towards showing that there exists an integer $f(k)$ such that $\alpha(G[A]) \leq f(k)-1$ through a number of claims (and therefore $G$ is $P_4+f(k)P_1$-free).
    \bigskip 

\noindent \textbf{Claim 1.} Let $z \in A$ and let $x \notin A \cup P$ be adjacent to $z$. Then $N_P(x) \in \{\{b,c\},\{a,b,c,d\}\}$.

\bigskip

Since $x \notin A \cup P$, we have that $N_P(x)$ is nonempty. Hence, there are $15$ possibilities to consider for how $N_P(x)$ looks. As summarized in Table~\ref{tab:PossibilitiesForNPx}, 13 out of 15 possibilities lead to a contradiction. This proves the claim.

\begin{table}[h]
\centering
\begin{tabular}{c|c}
\hline
$N_P(x)$ & Forbidden induced subgraph \\
\hline
$\{a\}$ &  $P_5$: $\{z,x,a,b,c\}$\\
$\{b\}$ &  $P_5$: $\{z,x,b,c,d\}$\\
$\{c\}$ &  $P_5$: $\{z,x,c,b,a\}$\\
$\{d\}$ &  $P_5$: $\{z,x,d,c,b\}$\\
$\{a,b\}$ &  $P_5$: $\{z,x,b,c,d\}$\\
$\{a,c\}$ &  $\text{chair}$: $\{x,z,a,c,d\}$\\
$\{a,d\}$ &  $P_5$: $\{z,x,a,b,c\}$\\
$\{b,c\}$ &  -\\
$\{b,d\}$ &  $\text{chair}$: $\{x,z,d,b,a\}$\\
$\{c,d\}$ &  $P_5$: $\{z,x,c,b,a\}$\\
$\{a,b,c\}$ &  $\text{chair}$: $\{x,z,a,c,d\}$\\
$\{a,b,d\}$ &  $\text{chair}$: $\{x,z,a,d,c\}$\\
$\{a,c,d\}$ &  $\text{chair}$: $\{x,z,d,a,b\}$\\
$\{b,c,d\}$ &  $\text{chair}$: $\{x,z,d,b,a\}$\\
$\{a,b,c,d\}$ & -\\
\hline
\end{tabular}
\caption{Summary of the 15 cases for $N_P(x)$ in the chair-free case.}
\label{tab:PossibilitiesForNPx}
\end{table}

\bigskip

\noindent

Define $T := \{x~\notin A \cup P~|~N_P(x)=\{b,c\}\}$ and $U := \{x~\notin A \cup P~|~N_P(x)=\{a,b,c,d\}\}$. Because of Claim 1, every vertex $x~\notin A \cup P$ that has a neighbor in $A$ is in either $T$ or $U$. We now focus on homogeneous sets.

\noindent \textbf{Claim 2.} Every connected component of $G[A]$ is a homogeneous set.

\bigskip

Let $H$ be a vertex set that induces a connected component of $G[A]$. Clearly, each vertex $y \in A \setminus H$ is anticomplete to $H$. If a vertex $y\in V(G)\setminus H$ is mixed on $H$, then since $G[H]$
is connected, there exists an edge $uv\in E(G[H])$ such that $y$ is adjacent
to exactly one of $u$ and $v$. Suppose for the sake of obtaining a contradiction that there exists a vertex $y \in V(G) \setminus A$ such that it distinguishes the edge $uv \in E(G[H])$, say $yu \in E(G)$ and $yv \notin E(G)$. We have $y \notin P$, because $H$ is anticomplete to $P$. Therefore, $y \in T \cup U$. If $y \in T$, then $\{a,b,y,u,v\}$ induces a $P_5$. If $y \in U$, then $\{y,a,c,u,v\}$ induces a $\text{chair}$. This leads to a contradiction and proves the claim.

\bigskip

\noindent

We will now focus on bounding the number of connected components of $G[A]$. Let $T_1, T_2, \ldots, T_r$ be all vertex sets that induce a connected component of $A$ that have some neighbor in $T$. We now prove the following claim.

\noindent \textbf{Claim 3.} We have $r \leq k-3$.

\bigskip

If there is a vertex $t \in T$, then $t$ is adjacent to at most one set $T_i$, for some integer $1 \leq i \leq r$. Indeed, if $t$ would be adjacent to $T_i$ and $T_j$ ($i \neq j$), then let $x \in T_i$ and $y \in T_j$. Now $\{t,x,y,c,d\}$ induces a chair.

For each integer $i$ for which $1 \leq i \leq r$, let $t_i \in T_i$ be a vertex that has a neighbor $t_i'\in T$. If there exist two distinct integers $i$ and $j$ such that $t_i't_j' \notin E(G)$, then $\{t_i,t_i',b,t_j',t_j\}$ induces a $P_5$. Hence, $\{b,c\} \cup \{t_i'~|~1 \leq i \leq r\}$ induces a clique. Since $G$ is a $k$-vertex-critical graph that contains an induced $P_4$, we have that $G$ is not itself a clique and therefore $r+2 \leq k-1$. This proves the claim.

\bigskip

\noindent

Let $U_1, U_2, \ldots, U_s$ be all vertex sets that induce a connected component of $A$ that have no neighbor in $T$. We have $N(U_i) \subseteq U$. We will now work towards showing that $s$ is bounded from above by a function of $k$. Let $p$ be an integer such that $1 \leq p<k$ and define $\mathcal{U}_p := \{U_i ~|~ 1 \leq i \leq s\text{ and }\chi(G[U_i])=p\}$. We now show that the neighborhoods of components in $\mathcal{U}_p$ form an antichain.

\noindent \textbf{Claim 4.} Let $1 \leq p <k$ be an integer. Then $\{N(U_i)~|~U_i \in \mathcal{U}_p\}$ forms an antichain.

\bigskip

Suppose for the sake of obtaining a contradiction that there exist distinct sets $U_i, U_j \in \mathcal{U}_p$ such that $N(U_i) \subseteq N(U_j)$. Since $U_i$ and $U_j$ induce a different connected component of $G[A]$, we have that $U_i$ is anticomplete to $U_j$. Since $N(U_i)\subseteq N(U_j)$, every vertex of $N(U_i)$ has a neighbor in $U_j$. As $U_j$ is homogeneous due to Claim 2, every vertex of $N(U_i)$ is complete to $U_j$. By the definition of $\mathcal{U}_p$, we have $\chi(G[U_i])=\chi(G[U_j])=p$. Hence, applying Lemma~\ref{lem:incompatible} proves the claim.
\bigskip

\noindent

For each vertex $u \in U$ and each integer $1 \leq p<k$, define the set 
$$I_p(u) := \{i~|~1 \leq i \leq s\text{ and }N(u) \cap U_i\text{ is nonempty and }U_i \in \mathcal{U}_p\},$$ 
indicating in which sets $U_i \in \mathcal{U}_p$ the vertex $u$ has a neighbor. By Claim 2, every vertex $u \in U$ is complete to $\bigcup_{i \in I_p(u)}U_i$. For each integer $1 \leq p<k$, let $U_p^* \subseteq U$ be a maximal set such that $I_p(a) \neq I_p(b)$ for all distinct $a, b \in U_p^*$. We now show the existence of a slightly modified antichain in comparison with Claim 4.

\noindent \textbf{Claim 5.} Let $1 \leq p <k$ be an integer. Then $\{N(U_i) \cap U_p^*~|~U_i \in \mathcal{U}_p\}$ forms an antichain.

\bigskip

By Claim 4, for each two distinct $U_i, U_j \in \mathcal{U}_p$, there exist vertices $u_i' \in N(U_i)$ and $u_j' \in N(U_j)$ such that $u_i' \notin N(U_j)$ and $u_j' \notin N(U_i)$. By the definition of $U_p^*$, there exist distinct vertices $u_i'' ,u_j''\in U_p^*$ such that $I_p(u_i')=I_p(u_i'')$ and  $I_p(u_j')=I_p(u_j'')$ and therefore $u_i'' \in N(U_i)$, $u_j'' \in N(U_j)$, $u_i'' \notin N(U_j)$ and $u_j'' \notin N(U_i)$. This means that $N(U_i) \cap U_p^*$ and $N(U_j) \cap U_p^*$ are incomparable. Hence, $\{N(U_i) \cap U_p^*~|~U_i \in \mathcal{U}_p\}$ forms an antichain and this proves the claim.
\bigskip

\noindent

We now bound $|U_p^*|$.

\noindent \textbf{Claim 6.} Let $p$ be an integer such that $1 \leq p <k$. We have $|U_p^*| \leq 2(k-1)$.

\bigskip

Suppose $u, v \in U_p^*$ are such that $uv \notin E(G)$. If $I_p(u)$ and $I_p(v)$ are incomparable, then choose $i \in I_p(u) \setminus I_p(v)$ and $j \in I_p(v) \setminus I_p(u)$ and let $u_i$ be a vertex in $U_i$ and $u_j$ be a vertex in $U_j$. Then $\{u_i,u,a,v,u_j\}$ induces a $P_5$. Therefore, whenever $u, v \in U_p^*$ are such that $uv \notin E(G)$, the sets $I_p(u)$ and $I_p(v)$ must be comparable.

Suppose $U_p^*$ contains an independent set $\{u_1,u_2,u_3\}$. Relabel these vertices (if necessary) such that $I_p(u_1) \subset I_p(u_2) \subset I_p(u_3)$ (note that these inclusions are strict because of the definition of $U_p^*$). Let $i \in I_p(u_3) \setminus I_p(u_2)$ and consider a vertex $u_i \in U_i$. Then $\{a,u_1,u_2,u_3,u_i\}$ induces a $\text{chair}$. Therefore, $\alpha(G[U_p^*]) \leq 2$. Since $G$ is $k$-vertex-critical, we have $\chi(G[U_p^*]) \leq k-1$ and therefore $|U_p^*| \leq 2(k-1)$. This proves the claim.

\bigskip

\noindent

By combining the previous claims appropriately, we can finally bound $\alpha(G[A])$ by a function of $k$.

\noindent \textbf{Claim 7.} There exists an integer $c_k$ that only depends on $k$ such that $$\alpha(G[A]) \leq c_k\left(k-3+(k-1)\binom{2k-2}{k-1}\right).$$

\bigskip

By Claim 2 and by Lemma~\ref{lem:homogeneous}, every connected component of $G[A]$ is $m$-vertex-critical for some integer $m$ for which $1 \leq m < k$. By the induction hypothesis, this means that such a connected component is also $P_4+f(m)P_1$-free and by Theorem~\ref{th:P4PlusLP1ChairFree} there are only finitely many such graphs. Therefore, there exists an integer $c_k$ such that $\alpha(G[H]) \leq c_k$ for each vertex set $H$ that induces a connected component of $G[A]$.

By Claim 3, we have $r \leq k-3$. For each integer $p$ such that $1 \leq p<k$, we have that $\{N(U_i) \cap U_p^*~|~U_i \in \mathcal{U}_p\}$ forms an antichain by Claim 5. Since $|U_p^*| \leq 2(k-1)$ by Claim 6, we conclude that $|\mathcal{U}_p| \leq \binom{2k-2}{k-1}$ by Sperner's theorem~\cite{S28}. Since $p$ can attain $k-1$ values, we also conclude that $s \leq (k-1)\binom{2k-2}{k-1}$. Therefore, we obtain $\alpha(G[A]) \leq c_k\left(k-3+(k-1)\binom{2k-2}{k-1}\right)$, as desired.

\end{proof}

\section{Vertex-critical $(P_5,\text{cricket})$-free graphs}
\label{sec:cricket}
Beaton and Cameron~\cite{BC26} also showed the following result.

\begin{theorem}[\cite{BC26}]
\label{th:P4PlusLP1CricketFree}
    For all integers $k \geq 1$ and $\ell \geq 0$, there are only finitely many $k$-vertex-critical $(P_4+\ell P_1,P_5,\text{cricket})$-free graphs.
\end{theorem}

Before introducing the main theorem of this section, we first require the following technical lemma. Here, $R_r(s)=R(s,s,\ldots,s)$ is the Ramsey number representing the smallest integer $n$ such that every colouring of the edges of
$K_n$ using $r$ colours contains a monochromatic copy of $K_s$.

\begin{lemma}
\label{lem:RamseyCricket}
Let $q \geq 1$ be an integer, let $X$ be a finite set and let $Y$ be a graph with $\chi(Y) \leq q$. For every vertex $y \in Y$, a set $I(y) \subseteq X$ is given such that the following conditions hold:
\begin{itemize}
    \item If $y,y' \in Y$ are nonadjacent, then $I(y)$ and $I(y')$ are comparable.
    \item If $y,y' \in Y$ are adjacent, then $|I(y) \setminus I(y')| \leq 1$ and $|I(y') \setminus I(y)| \leq 1$.
\end{itemize}
For every $x \in X$, define $S_x := \{y \in Y~|~x \in I(y)\}$. If $\{S_x~|~x \in X\}$ forms an antichain, then $|X| \leq R_{q^2}(4)-1$.
\end{lemma}
\begin{proof}
Consider a proper $q$-colouring of $Y$ and denote the corresponding colour classes by $Y_1, Y_2, \ldots, Y_q$. Because of the first condition, we may assume that for each $1 \leq i \leq q$ the vertices in $Y_i$ are ordered such that if $y$ comes before $y'$ in this ordering, then $I(y) \subseteq I(y')$. In what follows, we will treat $Y_i$ with respect to this ordering.

For every $x \in X$ and every integer $1 \leq i \leq q$, we define $S_x^i := S_x \cap Y_i$. Note that $S_x^i$ can be obtained by removing zero or more vertices from the beginning of $Y_i$, because if $y,y' \in Y_i$ and $I(y)\subseteq I(y')$ and $x \in I(y)$, then we also have $x \in I(y')$. We shall abbreviate this observation as $(\dagger)$.

Suppose for the sake of obtaining a contradiction that $|X| \geq R_{q^2}(4)$. Let us order the elements of $X$ in an arbitrary way. If $x$ and $x'$ are two distinct elements from $X$, we write $x<x'$ if $x$ comes before $x'$ in this ordering. Let $G$ be a complete graph with vertex set $X$. Let $x, x' \in X$ be two distinct elements such that $x<x'$. Recall that $S_x$ and $S_{x'}$ are incomparable, because $\{S_x~|~x \in X\}$ forms an antichain. Therefore, there exists an integer $i$ such that $S_x^i \setminus S_{x'}^i \neq \emptyset$ and an integer $j$ such that $S_{x'}^j \setminus S_x^j \neq \emptyset$. We colour the edge $xx' \in E(G)$ with colour $(i,j)$. Note that there are at most $q^2$ different colours.

Since $|X| \geq R_{q^2}(4)$, there are four vertices $x_1, x_2, x_3, x_4 \in V(G)$ such that each edge between these four vertices is coloured with the same colour. Let us call this colour $(i,j)$. By the definition of the colours, we have for every $1 \leq s < t \leq 4$ that $S_{x_s}^i \setminus S_{x_t}^i \neq \emptyset$ and $S_{x_t}^j \setminus S_{x_s}^j \neq \emptyset$. By observation $(\dagger)$, we have $S_{x_t}^i \subset S_{x_s}^i$ and also $S_{x_s}^j \subset S_{x_t}^j$. Hence, we have $S_{x_4}^i \subset S_{x_3}^i \subset S_{x_2}^i \subset S_{x_1}^i$ and $S_{x_1}^j \subset S_{x_2}^j \subset S_{x_3}^j \subset S_{x_4}^j$.

Consider a vertex $y \in S_{x_2}^i \setminus S_{x_3}^i$ and a vertex $y' \in S_{x_3}^j \setminus S_{x_2}^j$. Note that we have $x_1, x_2 \in I(y)$, but $x_3, x_4 \notin I(y)$ and $x_3, x_4 \in I(y')$, but $x_1, x_2 \notin I(y')$. Therefore we have $|I(y) \setminus I(y')| \geq 2$ and $|I(y') \setminus I(y)| \geq 2$. This contradicts the two conditions and proves the lemma.
\end{proof}

We are now ready to prove the main theorem of this section. Its proof follows the same strategy as the proof of Theorem~\ref{th:chair}, but the arguments that are necessary to make it work are different because of the constraints that arise from forbidding the $\text{cricket}$ graph instead of the $\text{chair}$ graph.

\begin{theorem}
\label{th:cricket}
    For all integers $k \geq 1$, there are only finitely many $k$-vertex-critical $(P_5,\text{cricket})$-free graphs.
\end{theorem}
\begin{proof}
    Our proof strategy is to show that for all integers $k \geq 1$, there exists an integer $f(k)$ such that every $k$-vertex-critical $(P_5,\text{cricket})$-free graph is also $P_4+f(k) P_1$-free. If we can show this, the theorem follows from Theorem~\ref{th:P4PlusLP1CricketFree}.

    We will prove this statement by induction on $k$. If $k \leq 2$, the statement is trivial. Now let $k \geq 3$ and assume the statement holds for all integers $1 \leq m<k$. Let $G$ be a $k$-vertex-critical $(P_5,\text{cricket})$-free graph. If $G$ is $P_4$-free, the statement trivially holds. 
    
    Otherwise, let $P := \{a,b,c,d\}$ be a 4-vertex subset of $V(G)$ such that $P$ induces the path $a-b-c-d$. Let $A := \{x \in V(G) \setminus P ~|~x\text{ is anticomplete to }P\}$. We will now work towards showing that there exists an integer $f(k)$ such that $\alpha(G[A]) \leq f(k)-1$ through a number of claims (and therefore $G$ is $P_4+f(k)P_1$-free).
    \bigskip

\noindent \textbf{Claim 1.} Let $z \in A$ and let $x \notin A \cup P$ be adjacent to $z$. Then 
$$N_P(x) \in \{\{a,c\},\{b,c\},\{b,d\},\{a,b,c\},\{b,c,d\}\}.$$

\bigskip

Since $x \notin A \cup P$, we have that $N_P(x)$ is nonempty. Hence, there are $15$ possibilities to consider for how $N_P(x)$ looks. As summarized in Table~\ref{tab:PossibilitiesForNPxCricket}, 10 out of 15 possibilities lead to a contradiction. This proves the claim.

\begin{table}[h]
\centering
\begin{tabular}{c|c}
\hline
$N_P(x)$ & Forbidden induced subgraph \\
\hline
$\{a\}$ & $P_5$: $\{z,x,a,b,c\}$\\
$\{b\}$ & $P_5$: $\{z,x,b,c,d\}$\\
$\{c\}$ & $P_5$: $\{z,x,c,b,a\}$\\
$\{d\}$ & $P_5$: $\{z,x,d,c,b\}$\\
$\{a,b\}$ & $P_5$: $\{z,x,b,c,d\}$\\
$\{a,c\}$ & -\\
$\{a,d\}$ & $P_5$: $\{z,x,a,b,c\}$\\
$\{b,c\}$ & -\\
$\{b,d\}$ & -\\
$\{c,d\}$ & $P_5$: $\{z,x,c,b,a\}$\\
$\{a,b,c\}$ & -\\
$\{a,b,d\}$ & $\text{cricket}$: $\{x,z,d,a,b\}$\\
$\{a,c,d\}$ & $\text{cricket}$: $\{x,z,a,c,d\}$\\
$\{b,c,d\}$ & -\\
$\{a,b,c,d\}$ & $\text{cricket}$: $\{x,z,d,a,b\}$\\
\hline
\end{tabular}
\caption{Summary of the 15 cases for $N_P(x)$ in the cricket-free case.}
\label{tab:PossibilitiesForNPxCricket}
\end{table}

\bigskip

\noindent

Define the following sets:
\begin{align*}
L &:= \{x \notin A \cup P~|~N_P(x)=\{a,c\}\},\\
M &:= \{x \notin A \cup P~|~N_P(x)=\{b,c\}\},\\
R &:= \{x \notin A \cup P~|~N_P(x)=\{b,d\}\},\\
L^+ &:= \{x \notin A \cup P~|~N_P(x)=\{a,b,c\}\},\\
R^+ &:= \{x \notin A \cup P~|~N_P(x)=\{b,c,d\}\}.
\end{align*}
Because of Claim 1, every vertex $x \notin A \cup P$ that has a neighbor in $A$ belongs to $L\cup M\cup R\cup L^+\cup R^+$. We now focus on homogeneous sets.

\noindent \textbf{Claim 2.} Every connected component of $G[A]$ is a homogeneous set.

\bigskip

Let $H$ be a vertex set that induces a connected component of $G[A]$. Clearly, each vertex $y \in A \setminus H$ is anticomplete to $H$. If a vertex $y\in V(G)\setminus H$ is mixed on $H$, then since $G[H]$
is connected, there exists an edge $uv\in E(G[H])$ such that $y$ is adjacent
to exactly one of $u$ and $v$. Suppose for the sake of obtaining a contradiction that there exists a vertex $y \in V(G) \setminus A$ such that it distinguishes the edge $uv \in E(G[H])$, say $yu \in E(G)$ and $yv \notin E(G)$. We have $y \notin P$, because $H$ is anticomplete to $P$. Therefore, $y \in L\cup M\cup R\cup L^+\cup R^+$. If $y \in L$, then $\{b,a,y,u,v\}$ induces a $P_5$. If $y \in M$, then $\{a,b,y,u,v\}$ induces a $P_5$. If $y \in R$, then $\{c,d,y,u,v\}$ induces a $P_5$. If $y \in L^+$, then $\{d,c,y,u,v\}$ induces a $P_5$. If $y \in R^+$, then $\{a,b,y,u,v\}$ induces a $P_5$. All cases lead to a contradiction and this proves the claim.

\bigskip

\noindent

We will now focus on bounding the number of connected components of $G[A]$. Let $B := M\cup L^+\cup R^+$. Let $B_1, B_2, \ldots, B_r$ be all vertex sets that induce a connected component of $G[A]$ that have some neighbor in $B$. We now prove the following claim.

\noindent \textbf{Claim 3.} We have $r \leq k-3$.

\bigskip

Note that every vertex in $B$ is adjacent to both $b$ and $c$. Moreover, if there is a vertex $u \in B$, then $u$ is adjacent to at most one set $B_i$, for some integer $1 \leq i \leq r$. Indeed, if $u$ would be adjacent to $B_i$ and $B_j$ ($i \neq j$), then let $x \in B_i$ and $y \in B_j$. Now $\{u,x,y,b,c\}$ induces a cricket.

For each integer $i$ for which $1 \leq i \leq r$, let $b_i \in B_i$ be a vertex that has a neighbor $b_i'\in B$. If there exist two distinct integers $i$ and $j$ such that $b_i'b_j' \notin E(G)$, then $\{b_i,b_i',b,b_j',b_j\}$ induces a $P_5$. Hence, $\{b,c\} \cup \{b_i'~|~1 \leq i \leq r\}$ induces a clique. Since $G$ is a $k$-vertex-critical graph that contains an induced $P_4$, we have that $G$ is not itself a clique and therefore $r+2 \leq k-1$. This proves the claim.

\bigskip

\noindent

We now consider the remaining connected components of $G[A]$ (i.e., the connected components that have no neighbor in $B$). Since $G$ is connected, every connected component of $G[A]$ has at least one neighbor outside of $A$. Hence, every remaining connected component has all its neighbors in $L \cup R$. We now focus on $L$ and $R$.

\bigskip

\noindent \textbf{Claim 4.} The set $L$ is complete to the set $R$.

\bigskip

Suppose for the sake of obtaining a contradiction that $\ell \in L$ and $r \in R$ are vertices such that $\ell r \notin E(G)$. Then
$\{\ell,a,b,r,d\}$ induces a $P_5$. This leads to a contradiction and proves the claim.

\bigskip

The next claim focuses on the interaction between connected components of $G[A]$ and the sets $L$ and $R$.

\noindent \textbf{Claim 5.} No connected component of $G[A]$ has both a neighbor in $L$ and a neighbor in $R$.

\bigskip

Let $H$ be a vertex set that induces a connected component of $G[A]$. Suppose for the sake of obtaining a contradiction that $x \in H$ has a neighbor $\ell \in L$ and $y \in H$ has a neighbor $r \in R$. By Claim 2, we also have that $x$ is a neighbor of $r$. By Claim 4, we have $\ell r \in E(G)$ and therefore $\{\ell,a,c,r,x\}$ induces a cricket. This leads to a contradiction and proves the claim.

\bigskip

\noindent

By Claim 5, every remaining connected component has all its neighbors in $L$ or all its neighbors in $R$. Let $L_1, L_2, \ldots, L_s$ be all vertex sets that induce a connected component of $G[A]$ that have no neighbor in $B$ such that for each $1 \leq i \leq s$, we have $N(L_i) \subseteq L$. Similarly, let $R_1, R_2, \ldots, R_t$ be all vertex sets that induce a connected component of $G[A]$ that have no neighbor in $B$ such that for each $1 \leq i \leq t$, we have $N(R_i) \subseteq R$. By symmetry, it suffices to bound $s$ from above by a function of $k$. 

Let $p$ be an integer such that $1 \leq p<k$ and define $\mathcal{L}_p := \{L_i ~|~ 1 \leq i \leq s\text{ and }\chi(G[L_i])=p\}$. We now show that the neighborhoods of components in $\mathcal{L}_p$ form an antichain.

\noindent \textbf{Claim 6.} Let $1 \leq p <k$ be an integer. Then $\{N(L_i)~|~L_i \in \mathcal{L}_p\}$ forms an antichain.

\bigskip

Suppose for the sake of obtaining a contradiction that there exist distinct sets $L_i, L_j \in \mathcal{L}_p$ such that $N(L_i) \subseteq N(L_j)$. Since $L_i$ and $L_j$ induce a different connected component of $G[A]$, we have that $L_i$ is anticomplete to $L_j$. Since $N(L_i)\subseteq N(L_j)$, every vertex of $N(L_i)$ has a neighbor in $L_j$. As $L_j$ is homogeneous due to Claim 2, every vertex of $N(L_i)$ is complete to $L_j$. By the definition of $\mathcal{L}_p$, we have $\chi(G[L_i])=\chi(G[L_j])=p$. Hence, applying Lemma~\ref{lem:incompatible} proves the claim.
\bigskip

\noindent

For each vertex $\ell \in L$ and each integer $1 \leq p<k$, define the set 
$$I_p(\ell) := \{i~|~1 \leq i \leq s\text{ and }N(\ell) \cap L_i\text{ is nonempty and }L_i \in \mathcal{L}_p\},$$ 
indicating in which sets $L_i \in \mathcal{L}_p$ the vertex $\ell$ has a neighbor. By Claim 2, every vertex $\ell \in L$ is complete to $\bigcup_{i \in I_p(\ell)}L_{i}$. For each integer $1 \leq p<k$, let $L_p^* \subseteq L$ be a maximal set such that $I_p(a) \neq I_p(b)$ for all distinct $a, b \in L_p^*$. We now show the existence of a slightly modified antichain in comparison with Claim 6.

\noindent \textbf{Claim 7.} Let $1 \leq p <k$ be an integer. Then $\{N(L_i) \cap L_p^*~|~L_i \in \mathcal{L}_p\}$ forms an antichain.

\bigskip

By Claim 6, for each two distinct $L_i, L_j \in \mathcal{L}_p$, there exist vertices $\ell_i' \in N(L_i)$ and $\ell_j' \in N(L_j)$ such that $\ell_i' \notin N(L_j)$ and $\ell_j' \notin N(L_i)$. By the definition of $L_p^*$, there exist distinct vertices $\ell_i'' ,\ell_j''\in L_p^*$ such that $I_p(\ell_i')=I_p(\ell_i'')$ and  $I_p(\ell_j')=I_p(\ell_j'')$ and therefore $\ell_i'' \in N(L_i)$, $\ell_j'' \in N(L_j)$, $\ell_i'' \notin N(L_j)$ and $\ell_j'' \notin N(L_i)$. This means that $N(L_i) \cap L_p^*$ and $N(L_j) \cap L_p^*$ are incomparable. Hence, $\{N(L_i) \cap L_p^*~|~L_i \in \mathcal{L}_p\}$ forms an antichain and this proves the claim.
\bigskip

\noindent

We now focus on how $I_p(\ell)$ and $I_p(\ell')$ are related for two vertices $\ell,\ell'\in L_p^*$.

\bigskip

\noindent \textbf{Claim 8.} Let $1 \leq p<k$ be an integer. If $\ell,\ell'\in L_p^*$ are nonadjacent, then $I_p(\ell)$ and $I_p(\ell')$ are comparable.

\bigskip

Suppose for the sake of obtaining a contradiction that there exist integers $i\in I_p(\ell)\setminus I_p(\ell')$ and $j\in I_p(\ell')\setminus I_p(\ell)$. Let $x_i\in L_i$ and $x_j\in L_j$. Now $\{x_i, \ell, a, \ell',x_j\}$ induces a $P_5$. This leads to a contradiction and proves the claim.

\bigskip

\noindent \textbf{Claim 9.} Let $1 \leq p<k$ be an integer. If $\ell,\ell'\in L_p^*$ are adjacent, then $|I_p(\ell) \setminus I_p(\ell')| \leq 1$ and $|I_p(\ell') \setminus I_p(\ell)| \leq 1$.

\bigskip

The two inequalities are symmetric, so we focus on proving the first one. Suppose for the sake of obtaining a contradiction that $I_p(\ell) \setminus I_p(\ell')$ contains two distinct integers $i$ and $j$. Let $x_i\in L_i$ and $x_j\in L_j$. Now $\{\ell,x_i,x_j,\ell',a\}$ induces a cricket. This leads to a contradiction and proves the claim.

\bigskip

\noindent

By combining the previous claims appropriately, we can finally bound $\alpha(G[A])$ by a function of $k$.

\noindent \textbf{Claim 10.} There exists an integer $c_k$ that only depends on $k$ such that 
$$\alpha(G[A]) \leq c_k\left(k-3+2(k-1)(R_{(k-1)^2}(4)-1)\right).$$

\bigskip

By Claim 2 and by Lemma~\ref{lem:homogeneous}, every connected component of $G[A]$ is $m$-vertex-critical for some integer $m$ for which $1 \leq m < k$. By the induction hypothesis, this means that such a connected component is also $P_4+f(m)P_1$-free and by Theorem~\ref{th:P4PlusLP1CricketFree} there are only finitely many such graphs. Therefore, there exists an integer $c_k$ such that $\alpha(G[H]) \leq c_k$ for each vertex set $H$ that induces a connected component of $G[A]$.

By Claim 3, we have $r \leq k-3$. Fix an integer $p$ such that $1 \leq p<k$. By combining Claim 7, Claim 8 and Claim 9, we can apply Lemma~\ref{lem:RamseyCricket} with $X=\{i~|~L_i\in\mathcal{L}_p\}$, $Y=G[L_p^*]$ and $q=k-1$. Therefore, we obtain $|\mathcal{L}_p|\leq R_{(k-1)^2}(4)-1$. Since $p$ can attain at most $k-1$ values, we obtain $s \leq (k-1)(R_{(k-1)^2}(4)-1)$. Similarly, we also obtain $t \leq (k-1)(R_{(k-1)^2}(4)-1)$. Hence, the number of connected components of $G[A]$ is at most $k-3+2(k-1)(R_{(k-1)^2}(4)-1)$. Since every connected component of $G[A]$ has independence number at most $c_k$, we obtain $\alpha(G[A]) \leq c_k\left(k-3+2(k-1)(R_{(k-1)^2}(4)-1)\right)$.

\end{proof}

\vspace{-0.5cm}
\section{Characterizations}
\label{sec:characterizations}
Let $k \geq 1$ be an integer and let $\mathcal{F}$ be a family of graphs. In~\cite{HKLSS10}, Ho\`{a}ng et al. presented a recursive algorithm having the property that, if it terminates, it outputs all (finitely many) $k$-vertex-critical $\mathcal{F}$-free graphs. Later, several papers~\cite{GS18,GJORS26,XJGH25a,XJGH25b} extended this algorithm by further generalizing it and improving its efficiency. We refer the interested reader to the aforementioned references for a complete description of this algorithm and to~\cite{J25} for a broader overview of computer-assisted 
methods in graph theory. We characterized all $5$-vertex-critical $(P_5,\text{chair})$-free graphs, all $5$-vertex-critical $(P_5,\text{cricket})$-free graphs and all $6$-vertex-critical $(P_5,\text{cricket})$-free graphs by running this algorithm (i.e., the algorithm terminated in all these cases). These graphs are summarized in Table~\ref{tab:counts} and also made available on House of Graphs~\cite{CDG23}:
\begin{center}
\url{https://houseofgraphs.org/meta-directory/critical-h-free}.
\end{center}
\begin{table}[htbp]
    \centering
    \footnotesize
    \setlength{\tabcolsep}{4pt}
    \renewcommand{\arraystretch}{1.15}
    \begin{tabular}{|c|C{0.26\textwidth}|C{0.26\textwidth}|C{0.26\textwidth}|}
        \hline
        $n$ &
        \makecell{\# $5$-vertex-critical\\ $(P_5,\text{chair})$-free\\ graphs} &
        \makecell{\# $5$-vertex-critical\\ $(P_5,\text{cricket})$-free\\ graphs} &
        \makecell{\# $6$-vertex-critical\\ $(P_5,\text{cricket})$-free\\ graphs} \\
        \hline
        5 & 1 & 1 & 0\\
        6 & 0 & 0 & 1\\
        7 & 1 & 1 & 0\\
        8 & 7 & 7 & 1\\
        9 & 193 & 191 & 7\\
        10 & 3 & 2 & 192\\
        11 & 0 & 0 & 19,473\\
        12 & 0 & 0 & 222\\
        13 & 0 & 0 & 7\\
        \hline
        Total & 205 & 202 & 19,903\\
        \hline
    \end{tabular}
    \vskip12pt
    \caption{An overview of the number of pairwise non-isomorphic $5$-vertex-critical $(P_5,\text{chair})$-free graphs, $5$-vertex-critical $(P_5,\text{cricket})$-free graphs and $6$-vertex-critical $(P_5,\text{cricket})$-free graphs.}
    \label{tab:counts}
\end{table}

\vspace{-0.6cm}
\section*{Acknowledgements}
\noindent The author is supported by a Postdoctoral Fellowship of the Research Foundation Flanders (FWO) with grant number 1222524N. The author is also grateful to Jan Goedgebeur for making the graphs from Section~\ref{sec:characterizations} available on House of Graphs.

	%\newpage	

\end{document}